\documentclass[12pt,twoside]{article}

\usepackage{amsmath,amsthm,amsfonts,amssymb,latexsym,epsfig}
\usepackage{enumerate}
\usepackage{hyperref}
\date{}
\DeclareMathOperator{\str}{\text{str}}

\begin{document}
	\centerline{\Large{\bf On stress of  a vertex in a graph}}
	\centerline{}
	\centerline{\bf {K. Bhargava$^\dag$, N.N. Dattatreya$^*$, and R. Rajendra$^\ddag$}}
	\centerline{Department of Mathematics}
	\centerline{Mangalore University}
	\centerline{Mangalagangotri-574 199, India.}
	\centerline{\it $^\dag$Email: kantila.bhargava@gmail.com}
	\centerline{\it $^\ddag$Email: rrajendrar@gmail.com}
	
	\centerline{}
	\centerline{$^*$Halemane, Brahmin Street}
	\centerline{Nadahalli, P\&T: Sorab}
	\centerline{Shivamogga-577 429, India.}
	\centerline{\it $^*$Email: dattatreya.nadahalli@gmail.com}
	\centerline{}
	\theoremstyle{definition}
	\newtheorem{Theorem}{Theorem}[section]
	\newtheorem{Proposition}[Theorem]{Proposition}
	\newtheorem{Definition}[Theorem]{Definition}
	\newtheorem{Corollary}[Theorem]{Corollary}
	\newtheorem{Lemma}[Theorem]{Lemma}
	\newtheorem{Example}[Theorem]{Example}
	\newtheorem{Problem}[Theorem]{Problem}
	\newtheorem{Remark}[Theorem]{Remark}

	\begin{abstract}	
		The stress of  a vertex in a graph is the number of geodesics passing through it (A. Shimbel, 1953). A graph is $k$-stress regular if stress of each of its vertices is $k$. In this paper, we investigate some results  and compute stress of vertices in some standard graphs and give a characterization of graphs with all vertices of  zero stress  except for one. Also we compute stress of vertices in graphs of diameter 2 and in the corona product $K_m \circ G$. Further we prove that any strongly regular graph is stress regular and characterize $k$-stress regular graphs for $k=0,1,2$.
	\end{abstract}
	

\section{Introduction}  For standard terminology and notion in graph theory, we follow the text-book of
Harary~\cite{hara}. 

Let $G=(V,E)$ be a graph (finite, undirected, simple). The number of edges in a path $P$ is its length $l(P)$. A shortest path between two vertices $u$ and $w$  in $G$ is called a $u-w$ geodesic. We say that a geodesic $P$ is passing through a vertex $v$ in $G$ if $v$  is an internal vertex of $P$.  The length of a longest geodesic in $G$ is called the diameter of $G$, denoted by $d(G)$. Eccentricity $e(v)$ of a vertex $v$ denotes the distance between $v$ and a vertex farthest from $v$. For any vertex $v$, its open neighborhood $N_G(v)$ (or simply $N(v)$) is the set of all vertices which are adjacent to $v$ and the closed neighborhood of $v$ is $N[v]=N(v) \cup \{v\}$. A vertex $v$ is called a simplicial vertex if $N(v)$ induces a complete subgraph. A maximal connected subgraph of $G$ is a component of $G$. A cut-vertex of $G$ is the one whose removal increases the number of components of $G$. A nonseparable graph is connected, nontrivial and has no cut-vertex. A block of $G$ is a maximal nonseparable subgraph. A graph is said to be vertex transitive if the automorphism group of $G$ acts transitively on $V(G)$. We denote the minimum of the degrees of vertices of $G$ by $\delta(G)$. The corona product $G_1 \circ G_2$ of two graphs $G_1$ and $G_2$ with disjoint vertex sets is defined as the graph $G$ obtained by taking one copy of $G_1$ and $|V(G_1)|$ number of copies of $G_2$ and then joining the $i$th vertex of $G_1$ to every vertex in the $i$th copy of $G_2$.

The concept of stress of a vertex in a graph was defined by Alfonso Shimbel~\cite{shimbel} in 1953. The concept has many applications in the study of biological networks, social networks etc. Some related works can be found in \cite{cold stress}, \cite{bionetwork}, \cite{indumathy}, \cite{centrality indices} and \cite{diff cent}. In this paper,  we study some of the properties of this concept.

In section $2$, followed by definitions, we make some observations,  characterize zero stress vertices and give a formula for the total stress of a graph. In section $3$, we compute stress of vertices in some standard graphs. In section $4$, we give a characterization of graphs with all vertices of  zero stress  except for one. In section 5, we study stress imposing and stress neutral vertices. We obtain  formulae for computing stress of vertices in graphs of diameter 2 and in the corona product $K_m \circ G$. Also we prove that strongly regular graphs are stress regular. In section 6, we characterize $k$-stress regular graphs for $k=1,2$.

\section{Stress of a vertex}

\begin{Definition}[Alfonso Shimbel\cite{shimbel}]\label{def-str}
	Let $G$ be a graph and $v$ be a vertex in $G$. The stress of  $v$, denoted by $\text{str}_G(v)$ or simply $\str(v)$, is defined as the number of  geodesics in $G$ passing through $v$.
\end{Definition}

\begin{Definition}[S. Arumugam]\label{str reg}
	A graph is said to be $k$-stress regular if all of its vertices have stress $k$.
\end{Definition}

\begin{Definition}\label{total stress}
	Let $G=(V,E)$ be a graph. The total stress of $G$, denoted by  $N_\text{str}(G)$, is defined as, 
	$$ N_\text{str}(G) = \sum_{v\in V} \text{str}(v).$$
\end{Definition}

\noindent {\bf Few Observations}: 
\begin{enumerate}
	\item For any vertex $v$ in a graph $G$, $ 0 \leq \text{str}(v) \leq N$, where $N$ is the number of geodesics of length at least 2 in $G$. 

	\item If there is no geodesic of length  at least 2 passing through a vertex $v$ in a graph $G$, 	then $\text{str}(v) =0$. Hence for any vertex $v$ in a complete graph $K_n$, we have $\str (v)=0$.
	
	\item If $\eta$ is an automorphism of a graph $G$ and $v$ is any vertex  in $G$, then $\text{str}(v) =\text{str}(\eta(v))$. Hence it follows that any vertex transitive graph is stress regular. However the converse is not true. The graph in the Figure~\ref{stress regular not regular} is a $2$-stress regular graph, but it is not vertex transitive since it is not even regular. The same example shows that a stress regular graph may not be regular.
	
	\item A regular graph may not be stress regular. The graph in the Figure~\ref{fig1} is a regular graph of degree 3, but it is not stress regular. (The numbers written near the vertices indicate the stress of the corresponding vertex.)
	
\end{enumerate}

\begin{figure}[h!]
	\unitlength 1mm 
	\linethickness{0.4pt}
	\ifx\plotpoint\undefined\newsavebox{\plotpoint}\fi 
	\begin{picture}(140,30)(-40,155)
	\put(7,170){\circle*{2}}
	
	\put(49,170){\circle*{2}}
	
	\put(7,170){\line(1,0){42}}
	
	\put(49,170){\line(1,-1){12}}
	\put(61,158){\circle*{2}}
	
	\put(61,182){\circle*{2}}
	\put(49,170){\line(1,1){12}}
	
	\put(73,182){\circle*{2}}
	\put(61,182){\line(1,0){12}}
	
	\put(73,158){\circle*{2}}
	\put(61,158){\line(1,0){12}}
	\put(73,158){\line(0,1){23}}
	\put(61,158){\line(1,2){12}}
	\put(61,182){\line(1,-2){12}}
	
	\put(-5,158){\circle*{2}}
	\put(-5,182){\circle*{2}}
	
	\put(-17,158){\circle*{2}}
	\put(-17,182){\circle*{2}}
	
	\put(7,170){\line(-1,-1){12}}
	\put(7,170){\line(-1,1){12}}
	\put(-5,182){\line(-1,0){12}}
	\put(-5,158){\line(-1,0){12}}
	
	\put(-5,182){\line(-1,-2){12}}
	\put(-5,158){\line(-1,2){12}}
	\put(-17,158){\line(0,1){23}}

	\put(49,166){\makebox(0,0)[cc]{$43$}}
	\put(7,166){\makebox(0,0)[cc]{$43$}}
	\put(0,181){\makebox(0,0)[cc]{$16$}}
	\put(0,157){\makebox(0,0)[cc]{$16$}}
	\put(56,181){\makebox(0,0)[cc]{$16$}}
	\put(56,157){\makebox(0,0)[cc]{$16$}}
	
	\put(-21,181){\makebox(0,0)[cc]{$1$}}
	\put(-21,157){\makebox(0,0)[cc]{$1$}}
	\put(77,181){\makebox(0,0)[cc]{$1$}}
	\put(77,157){\makebox(0,0)[cc]{$1$}}
	\end{picture}
	
	\caption{A regular graph which is not stress regular}
	\label{fig1}
\end{figure}
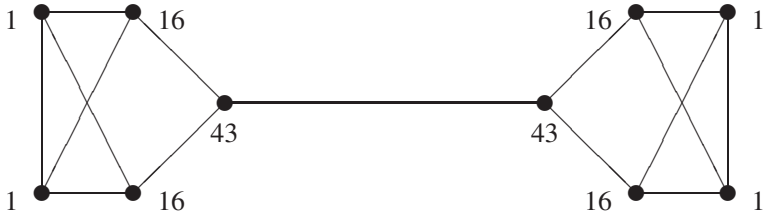

\begin{Theorem}\label{thm str 0}
	Let $G$ be any graph and let $v$ be any vertex in $G$. Then $\text{str}(v)=0$ if and only if $v$ is a simplicial vertex.
\end{Theorem}
\begin{proof}
	Let $\text{str}(v)=0$. If two neighbors $u$ and $w$ of $v$ are not adjacent, then $uvw$ is a geodesic passing through $v$, which will give a stress on $v$, a contradiction. If $v$ is a pendant vertex, clearly the result holds.
	
	Conversely, assume the given condition. Let $P$ be a $u-w$ path in $G$ through $v$. Then $P$ contains a $u-w$ sub-path of length one less than that of $P$, since any two neighbors of $v$ are adjacent. Thus $P$ cannot be a geodesic through $v$. Hence there are no geodesics passing through $v$ and thus $\str(v)=0$.
\end{proof}

The following corollary directly follows from the above theorem.
\begin{Corollary}\label{complete graph stress}
	A connected graph is $0$-stress regular if and only if it is a complete graph.
\end{Corollary}

\begin{Proposition}
	For any graph $G$ with diameter $d$, the total stress  of $G$, is given by
	\begin{equation*}\label{eqn.2}
		N_\text{str}(G) = \sum_{i=1}^{d} (i-1)f_i,
	\end{equation*} 
	where $f_i$ is the number of geodesics of length $i$ in $G$.
\end{Proposition}

\begin{proof}
	We know that, in a graph
	\begin{enumerate}[(i)]
		\item every geodesic of length $k$ has $k-1$ internal vertices;
		\item  every geodesic contributes one count to the stress of each of its internal vertices.
	\end{enumerate}
	Since the length of the longest geodesic in $G$ is $d$, the result follows from (i) and (ii).
\end{proof}

\section{Stress of vertices in some standard graphs}

The results in the following proposition follow by direct computations.
\begin{Proposition}\label{stress computation}
	\begin{enumerate}[(i)]
		
		\item In a complete bipartite graph $K_{mn}$,
		if $A$ and $B$ are the partite sets  with $|A|=m$ and $|B|=n$, then 
		$$\text{str}(v) = \begin{cases}
		\dfrac{n(n-1)}{2}, & \text{if}~v \in A \\
		\dfrac{m(m-1)}{2}, & \text{if}~v \in B.
		\end{cases} $$
		Consequently, $N_\text{str}(K_{m,n}) = \dfrac{mn}{2}(m+n-2)$.
		
		\item For every vertex $v$ in a cycle $C_{n}$ on $n$ vertices, $$\text{str}(v) = \begin{cases}
		\dfrac{(n-1)(n-3)}{8}, & \text{if $n$ is odd}\\
		\dfrac{n(n-2)}{8}, & \text{if $n$ is even}.			\end{cases} $$ Thus,  $$N_\text{str}(C_n) =  \begin{cases}
		\dfrac{n(n-1)(n-3)}{8}, & \text{if $n$ is odd}\\
		\dfrac{n^2(n-2)}{8}, & \text{if $n$ is even}.\end{cases} $$
		
		\item 	Let $Wd(n,m)$ denote the windmill graph~\cite{Gallian} constructed for $n \geq 2$ and $m \geq 2$ by joining $m$ copies of the complete graph $K_n$ at a shared universal vertex $v$. Then $\str(v)=m(m-1)(n-1)^2/2$
		and the rest of the vertices have zero stress. Hence $N_\text{str}(Wd(n,m))=m(m-1)(n-1)^2/2$.
	\end{enumerate}
\end{Proposition}

\begin{Proposition}
	Let $v$ be an internal vertex of a tree $T$ and let  $C_1,\dots, C_m$ be the components of $T-v$. Then $\text{str}(v)=\sum_{i<j} |C_i||C_j|$.
\end{Proposition}
\begin{proof}
	Since $T$ is a tree, if $v_i$ and $v_j$ are any two vertices in two different components  $C_i$ and $C_j$ of $T-v$, then there is a unique path $P$ in $T$ joining them. Clearly $v$ must be an internal vertex of $P$. Hence the pair $(v_i,v_j), \, i<j$ contributes exactly $1$ to the stress count of $v$. Thus the result follows. 
\end{proof}

\section{Some characterizations}

\begin{Theorem}\label{almost stressfree}
	Let $G=(V,E)$ be a connected graph with at least $3$ vertices. The graph $G$ has all vertices of zero stress except for one if and only if $G$ is a graph with a unique cut-vertex such that all its blocks are complete subgraphs of $G$.
\end{Theorem}
\begin{proof}
	$(\Rightarrow ):$  Suppose that the graph $G$ has all vertices of zero stress  except for one vertex, say $v$. Define a relation  $\sim$ on $V-\{v\}$ as follows: for $u, w \in V-\{v\}$,
	$ u\sim w \Longleftrightarrow u\in N[w]$.  
	We  show that $\sim$ is an equivalence relation on $V-\{v\}$. Clearly $\sim$ is reflexive and symmetric.
	
	We prove the transitivity of $\sim$: Suppose $u\sim w$ and $w\sim x$ in $V-\{v\}$. we need to show that $u \sim x$.  If $u=x$, then clearly $u\sim x$. If $u \neq x$, we need to show that $u$ and $x$ are adjacent. If either $u$ or $x$ is equal to $w$, then  $u$ and $x$ are adjacent. Assume that $u \neq w$ and $x \neq w$. If $u$ and $x$ are not adjacent, then $uwx$ is a geodesic of length $2$ passing through $w$, so that $\str(w)>0$. But then $w \ne v$  contradicts the hypothesis. Thus $u$ and $x$ are adjacent in $G$. So $u \sim x$.  
	
	Thus $\sim$ is an equivalence relation on $V-\{v\}$. Hence  $\sim$  partitions the set $V-\{v\}$ into disjoint equivalence classes, say $V_1$, \ldots, $V_k$. Note that $k>1$. If $k=1$, then $N(v)\subset V_1$, and so $N(v)$ induces a complete subgraph of $G$;  hence $\str(v)=0$, which is a contradiction. Clearly, by the definition of $\sim$, each $V_i$ induces a complete subgraph of $G$ and the induced subgraphs $\langle V_i\rangle$ and $\langle V_j\rangle$ for $i\neq j$ are mutually disjoint (vertex disjoint) subgraphs of $G$.
	
	We claim that $v$ is adjacent to all the vertices in $G$. Let $u \in  V-\{v\}$. Then  $u \in V_i$ for exactly one $i$. If $v$ is not adjacent to $u$, then there is a geodesic of length 2  from $v$ to $u$ passing through a vertex $w$ in $V_i$ (because $G$ is connected and $V_i$ induces a complete subgraph of $G$). This implies $\str(w)>0$, which is not possible. Hence $v$ is adjacent to $u$. 	This proves the claim.
	
	Now its clear that the removal of $v$ from $G$ disconnects the graph into components $\langle V_1 \rangle, \dots, \langle V_k \rangle$, where $k>1$. So $v$ is a cut-vertex in $G$. There is no other cut-vertex in $G$, because each vertex $u$ different from $v$ belongs to the complete subgraph induced by $V_i\cup \{v\}$ for exactly one $i$. Since $v$ is adjacent to all other vertices in $G$,
	$V_i \cup \{v\}$ induces a block that is a complete subgraph of $G$, for each $i$.

	$(\Leftarrow ):$ Conversely, suppose that $G$ is a  graph with a unique cut-vertex $v$ such that all its blocks are complete. Clearly, there are at least $2$ blocks. Let $u$ be a vertex in $G$, $u \neq v$. Then $u$ belong to a block in $G$. Since each block in $G$ is a complete subgraph, it follows that the neighbors of $u$ induce a complete subgraph of $G$ and so $\str(u) =0$. Let $B_1$ and $B_2$ be two distinct blocks in $G$. Let $x$ be a vertex in $B_1$ and $y$ be a vertex in $B_2$. Clearly, $x$ and $y$ are not adjacent. Since  $v$ is the unique cut-vertex in the connected graph $G$, every path between $x$ and $y$ must pass through $v$ and so there is a geodesic between $x$ and $y$ passing through $v$. Hence  $\str(v) > 0 $.~~~
	
\end{proof}

\begin{Corollary}
	Let $G$ be a connected  graph with $n+1$ vertices. Then $G \cong K_{1n}$ if and only if $G$ has exactly one vertex of stress $n(n-1)/2$, with the stress of remaining vertices being zero.
\end{Corollary}
\begin{proof}
	The direct part follows by computation.
	Conversely, let $v,w_1, \dots, w_n$ be the vertices of $G$, where we assume $\str(v)=n(n-1)/2$ and $\str(w_i)=0$ for all $i$. From the proof of the direct part of the Theorem~\ref{almost stressfree}, it follows that $\text{deg}(v)=n$. Now if $\text{deg}(w_i)>1$ for some $i$, then there exists a $j\ne i$ such that $w_i$ and $w_j$ are adjacent. Thus there is no stress on $v$ due to the pair $(w_i,w_j)$ and hence, the maximum possible stress on $v$ is $(n-2)+(n-2)+(n-3)+(n-4)+\dots+2+1$ which is less than $n(n-1)/2$, a contradiction. Hence $\text{deg}(w_i)=1$ for all $i$. Thus $G \cong K_{1n}$.
\end{proof}

\section{Stress imposing and stress neutral vertices}

\begin{Definition}
	A vertex $u$ in a graph $G$ is called a stress imposing vertex with respect to a vertex $v$ if there exists a vertex $w$ and a $u-w$ geodesic $P$ passing through $v$. A vertex $u$ which is not stress imposing with respect to $v$ is called a stress neutral vertex with respect to $v$. A pair $(x,y)$ of vertices is said to impose a stress on $v$ if there exists a $x-y$ geodesic passing through $v$. 
\end{Definition}

\begin{Lemma}\label{stress imposing vertex}
	Let $u$ and $v$ be any two vertices  in a connected graph $G$. Then $u$ is a stress imposing vertex with respect to $v$ if and only if there exists $w\in N(v)$ such that $d(u,w)=d(u,v)+1$. Consequently $u$ is  stress neutral with respect to $v$ if and only if $d(u,w) \leq d(u,v)$ for all $w \in N(v)$.
\end{Lemma}
\begin{proof}
	Suppose $u$ imposes a stress on $v$. Then there exists a vertex $w$ and a $u-w$ geodesic $P$ passing through $v$. Let $x$ be the vertex in $P$ which lies immediately next to $v$. Then $x \in N(v)$ and  $d(u,x)=d(u,v)+1$.
	
	Conversely suppose there exists $w\in N(v)$ such that $d(u,w)=d(u,v)+1$. Let $P$ be a $u-v$ geodesic. Then the $u-w$ path $P'$ formed using $P$ followed by $w$ is a $u-w$ geodesic, which passes through $v$ and hence $u$ imposes a stress on $v$.
\end{proof}

\begin{Corollary}\label{diameter 2 stress}
	For any vertex $v$ in a graph $G$ of diameter 2, $\str(v)$ equals the number of unordered pairs of non-adjacent vertices in $N(v)$.
\end{Corollary}
\begin{proof}
	Since $d(G)=2$, if $u$ imposes a stress on $v$, it follows from the Lemma~\ref{stress imposing vertex} that, $u \in N(v)$. Also if $u, w \in N(v)$, then the pair $(u,w)$ imposes exactly one stress on $v$ if and only if $u$ and $w$ are non-adjacent. Hence the result follows.
\end{proof}

Let $G$ be a regular graph with $v$ vertices and degree $k$. $G$ is said to be strongly regular~\cite{rcbose} if there exist integers $\lambda$ and $\mu$ such that any two adjacent vertices have $\lambda$ common neighbors and any two non-adjacent vertices have $\mu$ common neighbors. We write $G=\text{srg}(v, k, \lambda, \mu)$ in this case.

\begin{Corollary}\label{srg stress regular}
	Any strongly regular graph $G=\text{srg}(v, k, \lambda, \mu)$ is stress regular.
\end{Corollary}
\begin{proof}
	If $\mu=0$, then it a well-known result that $G$ is  a disjoint union of one or more equal-sized complete graphs and hence $G$ is stress regular by Corollary~\ref{complete graph stress}.
	
	If $\mu \neq 0$, then by definition, it follows that $d(G)=2$. Hence using the Corollary~\ref{diameter 2 stress}, for any vertex $v$ in $G$, we have $\str(v)=k(k-1-\lambda)/2$.
\end{proof}

\begin{Remark} \begin{enumerate}
		\item The converse of the Corollary~\ref{srg stress regular} is not true. The graph in the Figure~\ref{stress regular not regular} is a $2$-stress regular graph, but it is not strongly regular.
		
		\item For each natural number $k$, there is $k$-stress regular graph. For example  the graph $C_{2k+2}^k=\text{srg}(2k+2, 2k, 2k-2, 2k)$ is $k$-stress regular, by Corollary~\ref{srg stress regular}.
	\end{enumerate}
	
\end{Remark}

\begin{Proposition}
	Let $G$ be any graph with $n$ vertices and let $G'$ be the corona product $K_m \circ G$ of the complete graph $K_m$, $m\geq 2$ and $G$. Let $v \in V(G')$. If $v$ lies in the copy of $K_m$, then  $$ \str(v) = \frac{mn(m-1)(n+1)}{2}  $$ and if $v$ lies in a copy of $G$,  then $\str(v)$ equals the  number of unordered pairs of non-adjacent vertices in $N_G(v)$.
\end{Proposition}
\begin{proof} 
	The first part follows by computation. For the second, let $H_1, \ldots, H_m$ be the components of $G' - V(K_m)$ (the graph obtained from $G'$ by deleting the vertices in the copy of $K_m$).	Note that $H_i \cong G$, for all $i$.  Let $w \in V(H_i)$ for some $i$. Note that if $u \in V(G')-V(H_i)$, then $d_{G'}(u,x)=d_{G'}(u,w)$ for all $x \in N_{G'}(w)$, so all such vertices $u$ are stress neutral with respect to $w$, by Lemma~\ref{stress imposing vertex}. Now, note that if $x,y \in V(H_i)$, then $d_{G'}(x,y) \leq 2$ and hence any vertex $x \in V(H_i)-N_G(v)$, must be stress neutral with respect to $w$, by Lemma~\ref{stress imposing vertex}. Hence it follows that $\str(w)$ equals the number of unordered pairs of non-adjacent vertices in $N_G(w)$.
\end{proof}

\section{Stress regular graphs}

\begin{Lemma}\label{min degree in stress regular}
	Let $G$ be a $k$-stress regular connected graph with $k \geq 1$ and $d(G)=2$. Let $n$ be the smallest positive integer satisfying $k \leq \binom{n}{2}$. Then $\delta(G) \geq n$.
\end{Lemma}
\begin{proof}
	Suppose $d(v)=m<n$ for some vertex $v$. Since $d(G)=2$, the $\str(v)$ equals the number of unordered pairs of non-adjacent vertices in $N(v)$ and hence $\str(v) \leq \binom{m}{2} <k$ since $m<n$ and by the choice of $n$. The contradiction proves the result.
\end{proof}

\begin{Lemma}\label{geodesic length restriction}
	Let $G$ be any graph. Then for any vertex $v$ of a geodesic $P$ in $G$, we have $\str_P(v) \leq \str_G(v)$. Let $k=\text{max}\{\text{str}(v) : v \in V(G) \}$. If $n$ is a positive integer such that $\lfloor n/2 \rfloor \lceil n/2 \rceil > k$, then all geodesics in $G$ are of length less than $n$.
\end{Lemma}
\begin{proof}
	The first part follows by noting that any sub-path of a geodesic in $G$ is also a geodesic in $G$.	For the second part, suppose $P$ is a geodesic in $G$ with $m=l(P) \geq n$. If $w$ is a vertex in the center of $P$, then $\text{str}_P(w)=\lfloor m/2 \rfloor \lceil m/2 \rceil \geq \lfloor n/2 \rfloor \lceil n/2 \rceil > k$.  By the first part, we have $\text{str}_P(w) \leq \text{str}_G(w)$. Hence  $\text{str}_G(w)>k$, a contradiction. Hence the result follows. \end{proof}

\begin{Lemma}\label{eccentricity 1 lemma}
	Let $G$ be a connected graph which is not complete and let $v, w$ be any two vertices in $G$, with $e(v)=1$. Then $\str(w) \leq \str(v)$, where equality holds if and only if $e(w)=1$. Further if $\str(v) \geq 1$, then there exist a vertex $u$ in $G$ for which $\str(u) < \str(v)$.
\end{Lemma}
\begin{proof}
	Since $e(v)=1$ and $G$ is connected, it follows that $v$ must be adjacent to all the vertices in $G$. Now since $G$ is not complete, we must have $d(G)=2$. Hence  whenever a pair $(x,y)$ of vertices impose a stress on $w$, they impose a stress on $v$ too. Hence $\str(w)\leq \str(v)$. Now suppose $\str(w) = \str(v)$. If $w$ is not adjacent to a vertex $u$ in $G$, then the pair $(u,w)$ imposes a stress on $v$, but not on $w$. Consequently $\str(w) < \str(v)$, a contradiction. Thus $e(w)=1$. Conversely if $e(w)=1=e(v)$, then $\str(v) \leq \str(w) \leq \str(v)$ from the first part and so the equality.
	
	For the second part, since $\str(v) \geq 1$, there exist a pair $(u,x)$ of non adjacent vertices in $G$. Then $e(u)>1$. So using the first part, we have $\str(u) < \str(v)$.
\end{proof}

\begin{Corollary}\label{e(v)>1 in k stress regular}
	Lt $G$ be a $k$-stress regular connected graph with $k \geq 1$. Then $e(v) >1$ for any vertex $v$ in $G$. 
\end{Corollary}
\begin{proof}
	Since $k \geq 1$, by Corollary~\ref{complete graph stress}, it follows that $G$ is not complete. Further $G$ is connected implies $e(v) \geq 1$ for any vertex $v$. Suppose $e(v)=1$ for some vertex, then by the second part of the Lemma~\ref{eccentricity 1 lemma}, there exists a vertex $w$ such that $\str(w) < \str(v)$, contradicting the assumption. Hence the result.
\end{proof}

\begin{Theorem}
	A connected graph $G$ is $1$-stress regular if and only if it is isomorphic to the cycle $C_4$ or $C_5$.
\end{Theorem}
\begin{proof}
	The converse follows from Proposition~\ref{stress computation}(ii). We need only prove the direct part.
	
	Let $G$ be a connected graph which is $1$-stress regular. By Lemma~\ref{geodesic length restriction}, we observe that all geodesics in $G$ are of length less than $3$. Hence $e(v) \leq 2$ for all $v\in V(G)$. But by Corollary~\ref{e(v)>1 in k stress regular}, we have $e(v)>1$ for all $v\in V(G)$. Hence we must have $e(v)=2$ for all $v\in V(G)$. Also by Lemma~\ref{min degree in stress regular}, we must have $d(v)\geq 2$ for all $v\in V(G)$. We claim that $d(v)=2$ for all $v\in V(G)$. 
	
	Suppose $d(v)>2$ for some vertex $v$ of $G$. Since $d(G)=2$, only the vertices in the neighborhood of $v$ can impose stress on $v$, by Lemma~\ref{stress imposing vertex}. Now since $\text{str}(v)=1$, there exist two vertices $u$ and $w$ in the neighborhood of $v$ which are not adjacent to each other. Since $d(v)>2$, we can find another vertex $a$ which is adjacent to $v$. But then $a$ must be stress neutral, so we must have that $a$ is adjacent to every vertex in the neighborhood of $v$. Now since $e(v)=2$, we can find a vertex $x$ which is at a distance 2 from $v$. If $x$ is adjacent to $a$, then there will be at least two geodesics passing through $a$, namely $uaw$ and $xav$, which is not possible since $G$ is 1-stress regular.  Similarly $x$ cannot be adjacent to any other vertex in the neighborhood of $v$ other than $u$ and $w$. Hence $x$ must be adjacent to either $u$ or to $w$, say to $u$. But then there will be at least two different geodesics passing through $u$, namely $xuv$ and $xua$. Again this is not possible. Hence $d(v)=2$ for all $v\in V(G)$. 
	
	Thus $G$ is a connected 2-regular graph and hence $G$ must be a cycle. Now by Proposition~\ref{stress computation}(ii), it follows that $G$ must be isomorphic to either $C_4$ or $C_5$.
\end{proof}

\begin{Theorem}
	A connected graph $G$ is $2$-stress regular if and only if it is isomorphic to one of the graphs given by the Figures~\ref{stress regular not regular} or \ref{2 stress regular 3 regualr} or \ref{2 stress regualr 4 regular}.
 \end{Theorem}

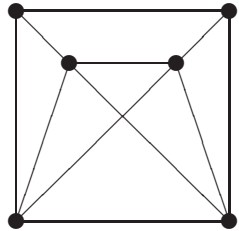
\begin{figure}[h!]
	\unitlength 1mm 
	\linethickness{0.4pt}
	\ifx\plotpoint\undefined\newsavebox{\plotpoint}\fi 
	\begin{picture}(140,30)(-40,190)
	
	\put(14,194){\circle*{2}}
	\put(42,194){\circle*{2}}
	\put(14,222){\circle*{2}}
	\put(42,222){\circle*{2}}
	
	\put(14,194){\line(0,1){28}}
	\put(14,194){\line(1,0){28}}
	\put(42,194){\line(0,1){28}}
	\put(14,222){\line(1,0){28}}
	
	\put(14,194){\line(1,1){28}}
	\put(14,222){\line(1,-1){28}}
	
	\put(14,194){\line(1,3){7}}
	\put(42,194){\line(-1,3){7}}
	
	\put(21,215){\line(1,0){14}}
	
	\put(21,215){\circle*{2}}
	\put(35,215){\circle*{2}}
	
	\end{picture}
	\caption{A 2-stress regular graph which is not regular}
	\label{stress regular not regular}
\end{figure}

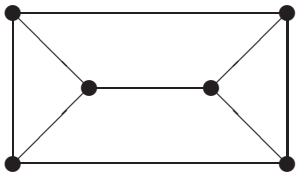
\begin{figure}[h!]
	\unitlength 1mm 
	\linethickness{0.4pt}
	\ifx\plotpoint\undefined\newsavebox{\plotpoint}\fi 
	\begin{picture}(120,30)(-35,190)
	
	\put(14,194){\circle*{2}}
	\put(50,194){\circle*{2}}
	\put(14,214){\circle*{2}}
	\put(50,214){\circle*{2}}
	
	\put(14,194){\line(0,1){20}}
	\put(14,194){\line(1,0){36}}
	\put(50,194){\line(0,1){20}}
	\put(14,214){\line(1,0){36}}
	
	\put(24,204){\circle*{2}}
	\put(40,204){\circle*{2}}
	\put(24,204){\line(1,0){15}}
	
	\put(14,194){\line(1,1){10}}
	\put(50,194){\line(-1,1){10}}
	\put(14,214){\line(1,-1){10}}
	\put(50,214){\line(-1,-1){10}}

	\end{picture}
	\caption{A 2-stress regular graph which is 3 regular}
	\label{2 stress regular 3 regualr}
\end{figure}

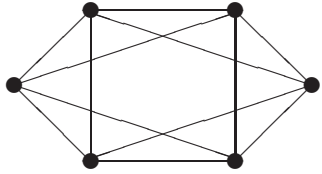
\begin{figure}[h!]
	\unitlength 1mm 
	\linethickness{0.4pt}
	\ifx\plotpoint\undefined\newsavebox{\plotpoint}\fi 
	\begin{picture}(140,30)(-44,190)
	
	\put(14,194){\circle*{2}}
	\put(33,194){\circle*{2}}
	\put(14,214){\circle*{2}}
	\put(33,214){\circle*{2}}
	
	\put(14,194){\line(1,0){19}}
	\put(14,214){\line(1,0){19}}
	
	\put(4,204){\circle*{2}}
	\put(43,204){\circle*{2}}
	
	\put(14,194){\line(-1,1){10}}
	\put(33,194){\line(1,1){10}}
	\put(14,214){\line(-1,-1){10}}
	\put(33,214){\line(1,-1){10}}
	
	\put(4,204){\line(3,1){30}}
	\put(4,204){\line(3,-1){30}}
	\put(43,204){\line(-3,1){30}}
	\put(43,204){\line(-3,-1){30}}
	\put(14,214){\line(0,-1){20}}
	\put(33,214){\line(0,-1){20}}

	\end{picture}
	\caption{A 2-stress regular graph which is 4  regular}
	\label{2 stress regualr 4 regular}
\end{figure}

\begin{proof}
	The converse follows from direct computations. We prove the direct part.
	
	Let $G$ be a connected graph which is $2$-stress regular. By Lemma~\ref{geodesic length restriction}, we observe that all geodesics in $G$ are of length less than $4$. We claim that $G$ does not contain a geodesic of length 3.
	
	Suppose $F:uabv$ is a geodesic of length 3 in $G$. Note that $\text{str}_F(a)=\text{str}_F(b)=2$. We claim that $d(a)=d(b)=2$. For suppose a vertex $x$ is adjacent to $a$, where $x\neq u, b$. Then, since $G$ is $2$-stress regular and there are two geodesics passing through $a$ already, we must have that $x$ must be stress neutral with respect to $a$. Then by Lemma~\ref{stress imposing vertex}, it follows that $x$ must be adjacent to both $u$ and $b$. But again, by similar argument, we must have that $x$ must be adjacent to $v$. But then $uxv$ is a path of length 2 joining $u$ and $v$, contradicting the fact that $F$ is a $u-v$ geodesic. Thus we must have $d(a)=2$. Similarly $d(b)=2$. Now $u$ cannot be adjacent to $v$ and $u$ cannot be a pendant vertex in $G$ too. So we can find a vertex $x\neq a$ such that $x$ is adjacent to $u$. Note that since $F$ is a $u-v$ geodesic, $x$ cannot be adjacent to $v$. Now since $d(b)=2$, $x$ cannot be adjacent to $b$. So it follows that $xuab$ is a $x-b$ geodesic. But then there will be at least 3 different geodesics passing through $a$, a contradiction. Thus $G$ does not contain a geodesic of length 3.
	
	Hence all geodesics in $G$ are of length less than $3$. Hence $e(v) \leq 2$ for all $v\in V(G)$. But by Corollary~\ref{e(v)>1 in k stress regular}, we have $e(v)>1$ for all $v\in V(G)$. Hence we must have $e(v)=2$ for all $v\in V(G)$. Also by Lemma~\ref{min degree in stress regular}, we must have $d(v)\geq 3$ for all $v\in V(G)$. We claim that $d(v)\leq 4$ for all $v\in V(G)$.
	
	Suppose $d(v)>4$ for some vertex $v$ of $G$, say $d(v)=5$. Since $e(u)=2$ for all vertices $u$ in G, only the vertices in the neighborhood of $v$ can impose stress on $v$, by Lemma~\ref{stress imposing vertex}.
	
	Suppose there is no vertex in the neighborhood of $v$ imposing more than one stress on $v$, then the arrangement of vertices in the neighborhood of $v$ is as shown in Figure~\ref{d(v)=5 & no vertex causing >1 stress}, where $a$ is not adjacent to $b$ and $c$ is not adjacent to $d$. Consider a vertex $x$ of $G$ which is at a distance 2 from $v$. If $x$ is adjacent to $u$, then $aub$, $cud$ and $xuv$ will be 3 distinct geodesics passing through $u$, a contradiction. Hence $x$ cannot be adjacent to $u$. So $x$ must be adjacent to one of the other 4 vertices, say to $a$. Then $cad$, $xav$ and $xau$ will be 3 distinct geodesics passing through $a$, a contradiction. Similar contradictions are obtained if $x$ is adjacent to $b$ or $c$ or $d$.
	
	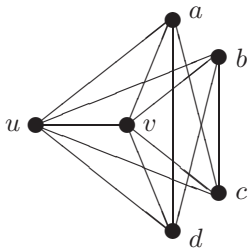
\begin{figure}[h!]
		\unitlength 1mm 
		\linethickness{0.4pt}
		\ifx\plotpoint\undefined\newsavebox{\plotpoint}\fi 
		\begin{picture}(140,35)(-40,180)
		
		\put(15,194){\circle*{2}}
		\put(27,194){\circle*{2}}
		\put(33,208){\circle*{2}}
		\put(33,180){\circle*{2}}
		\put(39,203){\circle*{2}}
		\put(39,185){\circle*{2}}

		\put(15,194){\line(1,0){12}}
		\put(27,194){\line(2,5){6}}
		\put(27,194){\line(5,4){11}}
		\put(27,194){\line(5,-4){11}}
		\put(27,194){\line(2,-5){6}}
		\put(33,180){\line(0,1){27}}
		\put(15,194){\line(5,4){18}}
		\put(15,194){\line(5,-4){18}}
		\put(39,185){\line(0,1){18}}
		\put(39,185){\line(-1,4){6}}
		\put(39,203){\line(-1,-4){6}}
		
		\put(15,194){\line(5,2){24}}
		\put(15,194){\line(5,-2){24}}
		
		\put(30,194){\makebox(0,0)[cc]{$v$}}
		\put(12,194){\makebox(0,0)[cc]{$u$}}
		\put(42,203){\makebox(0,0)[cc]{$b$}}
		\put(42,185){\makebox(0,0)[cc]{$c$}}
		\put(36,179){\makebox(0,0)[cc]{$d$}}
		\put(36,209){\makebox(0,0)[cc]{$a$}}
		
		\end{picture}
		\caption{$d(v)=5$ and no vertex imposes more than one stress on $v$}
		
		\label{d(v)=5 & no vertex causing >1 stress}
	\end{figure}

	Suppose there is a vertex in the neighborhood of $v$ imposing more than one stress on $v$, then the arrangement of vertices in the neighborhood of $v$ is as shown in Figure~\ref{d(v)=5 & there is a vertex causing >1 stress}, where $x$ is not adjacent to $y$ and $z$.  Consider a vertex $u$ of $G$ which is at a distance 2 from $v$. If $u$ is adjacent to $a$, then $xay$, $xaz$ and $uav$ will be 3 distinct geodesics passing through $a$, a contradiction. Hence $u$ cannot be adjacent to $a$. Similarly, $u$ cannot be adjacent to $b$. So $u$ must be adjacent to one of the other 3 vertices, say to $x$. Then $uxa$, $uxb$ and $uxv$ will be 3 distinct geodesics passing through $x$, a contradiction. Similar contradictions are obtained if $u$ is adjacent to $y$ or  to $z$.

	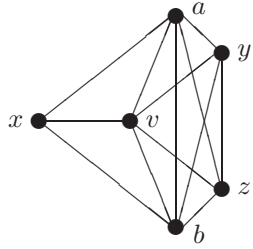
\begin{figure}[h!]
		\unitlength 1mm 
		\linethickness{0.4pt}
		\ifx\plotpoint\undefined\newsavebox{\plotpoint}\fi 
		\begin{picture}(140,35)(-40,180)
		
		\put(15,194){\circle*{2}}
		\put(27,194){\circle*{2}}
		\put(33,208){\circle*{2}}
		\put(33,180){\circle*{2}}
		\put(39,203){\circle*{2}}
		\put(39,185){\circle*{2}}

		\put(15,194){\line(1,0){12}}
		\put(27,194){\line(2,5){6}}
		\put(27,194){\line(5,4){11}}
		\put(27,194){\line(5,-4){11}}
		\put(27,194){\line(2,-5){6}}
		\put(33,180){\line(0,1){27}}
		\put(15,194){\line(5,4){18}}
		\put(15,194){\line(5,-4){18}}
		\put(39,185){\line(0,1){18}}
		\put(39,185){\line(-1,-1){6}}
		\put(39,203){\line(-1,1){6}}
		\put(39,185){\line(-1,4){6}}
		\put(39,203){\line(-1,-4){6}}
		
		\put(30,194){\makebox(0,0)[cc]{$v$}}
		\put(12,194){\makebox(0,0)[cc]{$x$}}
		\put(42,203){\makebox(0,0)[cc]{$y$}}
		\put(42,185){\makebox(0,0)[cc]{$z$}}
		\put(36,179){\makebox(0,0)[cc]{$b$}}
		\put(36,209){\makebox(0,0)[cc]{$a$}}
		
		\end{picture}
		\caption{$d(v)=5$ and there is a vertex imposing more than one stress on $v$}
		\label{d(v)=5 & there is a vertex causing >1 stress}
		
	\end{figure}
	
	Thus we cannot have $d(v)=5$. Similarly $d(v) \geq 6$ is not possible. Hence it follows that $d(v) \leq 4$ for all $v \in V(G)$. Now we complete the proof by considering Case 1 and Case 2 below:
	\vskip .3cm
	
	\noindent  \textbf{Case 1.} Suppose $d(v)=4$ for some vertex $v$ in $G$. Let the vertices adjacent to $v$ be $a, b, c, d$. Suppose none of these impose more than one stress on $v$. Then  the arrangement of vertices in the neighborhood of $v$ is as shown in Figure~\ref{d(v)=4 & no vertex causing >1 stress}, where $a$ is not adjacent to $c$ and $b$ is not adjacent to $d$.
	
	\begin{figure}[h!]
		\unitlength 1mm 
		\linethickness{0.4pt}
		\ifx\plotpoint\undefined\newsavebox{\plotpoint}\fi 
		\begin{picture}(140,30)(-44,190)
		
		\put(14,194){\circle*{2}}
		\put(33,194){\circle*{2}}
		\put(14,214){\circle*{2}}
		\put(33,214){\circle*{2}}

		\put(14,194){\line(1,0){19}}
		
		\put(14,214){\line(1,0){19}}
		
		\put(4,204){\circle*{2}}

		\put(14,194){\line(-1,1){10}}
		\put(14,214){\line(-1,-1){10}}
		
		\put(4,204){\line(3,1){30}}
		\put(4,204){\line(3,-1){30}}
		
		\put(14,214){\line(0,-1){20}}
		\put(33,214){\line(0,-1){20}}
		\put(1,204){\makebox(0,0)[cc]{$v$}}
		
		\put(33,217){\makebox(0,0)[cc]{$a$}}
		\put(14,191){\makebox(0,0)[cc]{$c$}}
		\put(33,191){\makebox(0,0)[cc]{$d$}}
		\put(14,217){\makebox(0,0)[cc]{$b$}}

		\end{picture}
		\caption{$d(v)=4$ and no vertex imposes more than one stress on $v$}
		
		\label{d(v)=4 & no vertex causing >1 stress}
	\end{figure}
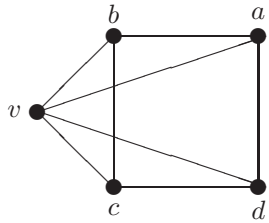
	
	Let $u$ be a vertex in $G$ which is at a distance 2 from $v$. Then $u$ must be adjacent to either $a$ or $b$ or $c$ or $d$, say to $a$. Now since $uav$ and $bad$ are 2 distinct geodesics passing through $a$ and since $\text{str}_G(a) =2$, we must have that $u$ must be adjacent to both $b$ and $d$, by Lemma~\ref{stress imposing vertex}. But then $ubv$ and $abc$ will be 2 distinct geodesics through $b$ and since $\text{str}_G(b) =2$, we must have that $u$ must be adjacent to $c$, by Lemma~\ref{stress imposing vertex}. Hence we obtain a 2-stress regular   subgraph $H$  of $G$ which is isomorphic to the graph given by the Figure~\ref{2 stress regualr 4 regular}. Note that, since $G$ is connected and $d(v)\leq 4$ for all $v\in V(G)$, it follows that $V(G)=V(H)$. Clearly $H$ is the maximal subgraph with the vertex set $V(H)$. Hence we must have $G=H$.
	
	Now suppose one of the vertices in the neighborhood of $v$, say $a$,  imposes more than one stress on $v$. Then $a$ must impose a stress count of 2 on $v$ and after some relabeling, we may assume that the arrangement of vertices in the neighborhood of $v$ is as shown in the Figure~\ref{d(v)=4 & there is a vertex causing >1 stress}, where $a$ is not adjacent to $b$ and  $c$.

	\begin{figure}[h!]
		\unitlength 1mm 
		\linethickness{0.4pt}
		\ifx\plotpoint\undefined\newsavebox{\plotpoint}\fi 
		\begin{picture}(140,30)(-44,190)
		\put(33,194){\circle*{2}}		
		\put(33,214){\circle*{2}}
		\put(14,204){\circle*{2}}		
		\put(33,214){\line(0,-1){20}}
		\put(4,207){\makebox(0,0)[cc]{$a$}}
		
		\put(14,207){\makebox(0,0)[cc]{$v$}}
		\put(11,193){\makebox(0,0)[cc]{$d$}}
		\put(36,194){\makebox(0,0)[cc]{$c$}}
		\put(36,214){\makebox(0,0)[cc]{$b$}}
		
		\put(14,204){\line(2,1){20}}
		\put(14,204){\line(2,-1){20}}
		\put(4,204){\circle*{2}}
		\put(14,194){\circle*{2}}
		\put(14,194){\line(1,1){19}}
		\put(14,194){\line(-1,1){10}}
		\put(14,194){\line(1,0){20}}
		\put(14,194){\line(0,1){10}}
		\put(4,204){\line(1,0){10}}
		\end{picture}
		\caption{$d(v)=4$ and there is a vertex imposing more than one stress on $v$}
		
		\label{d(v)=4 & there is a vertex causing >1 stress}
	\end{figure}
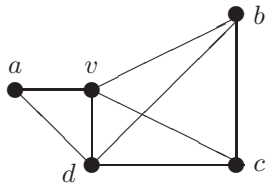
	
	Since $e(v)=2$, we can find a vertex $u$ which is at a distance 2 from $v$. Note that $u$ cannot be adjacent to $d$, since $G$ does not contain a vertex of degree 5. We prove that $u$ must be adjacent to $a$, $b$ and $c$. Note that if $u$ is adjacent to either $b$ or $c$, then by Lemma~\ref{stress imposing vertex}, it must be adjacent to both. 
	
	Suppose $u$ is adjacent to $a$ and that it is not adjacent to both $b$ and $c$. Since all geodesics in  $G$ are of length less than 3, we can find a $u-b$ geodesic of length 2, say $uxb$. Note that $x\ne a,v,d,c$ .  Further, since $G$ does not contain any vertex of degree 5, $x$ cannot be adjacent to $v$ and $d$. By Lemma~\ref{stress imposing vertex}, it follows that $x$ must be adjacent to $c$. Also note that, since  there are 2 distinct geodesics passing through $a$ namely $uav$ and $uad$, it follows that $x$ cannot be adjacent to $a$. The situation is shown in the Figure~\ref{u must be adj to b}.

	\begin{figure}[h!]
		\unitlength 1mm 
		\linethickness{0.4pt}
		\ifx\plotpoint\undefined\newsavebox{\plotpoint}\fi 
		\begin{picture}(140,30)(-48,190)
		\put(33,194){\circle*{2}}		
		\put(33,214){\circle*{2}}
		\put(14,204){\circle*{2}}		
		\put(33,214){\line(0,-1){20}}
		\put(1,204){\makebox(0,0)[cc]{$a$}}
		
		\put(14,207){\makebox(0,0)[cc]{$v$}}
		\put(11,193){\makebox(0,0)[cc]{$d$}}
		\put(36,194){\makebox(0,0)[cc]{$c$}}
		\put(36,214){\makebox(0,0)[cc]{$b$}}
		
		\put(14,204){\line(2,1){20}}
		\put(14,204){\line(2,-1){20}}
		\put(4,204){\circle*{2}}
		\put(14,194){\circle*{2}}
		\put(14,194){\line(1,1){19}}
		\put(14,194){\line(-1,1){10}}
		\put(14,194){\line(1,0){20}}
		\put(14,194){\line(0,1){10}}
		\put(4,204){\line(1,0){10}}
		\put(4,214){\circle*{2}}	
		\put(13,214){\circle*{2}}	
		\put(33,194){\line(-1,1){20}}
		\put(4,214){\line(0,-1){10}}
		\put(4,214){\line(1,0){28}}
		\put(1,214){\makebox(0,0)[cc]{$u$}}
		\put(13,217){\makebox(0,0)[cc]{$x$}}
		
		\end{picture}
		\caption{A figure to illustrate a situation in Case 1}
		
		\label{u must be adj to b}
	\end{figure}
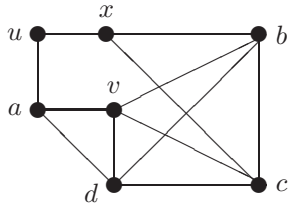
	
	Since $3 \leq d(p) \leq 4$ for every vertex $p$ in $G$, and since $\text{str}_G(u)=2$,  we can find a vertex $y$ in the neighborhood of $u$ different from $a$ and $x$, having the property that $y$ is adjacent to either $a$ or  $x$, but not both. Since $u$ is not adjacent to $v,b,c,d$, we have that $y \neq v,b,c,d$. If $y$ is adjacent to $a$, then by Lemma~\ref{stress imposing vertex}, it follows that $y$ must be adjacent to $v$ and to $d$, which is not possible, since $d(v)=4$. If $y$ is adjacent to $x$, then by Lemma~\ref{stress imposing vertex}, it follows that $y$ must be adjacent to $b$ and to $c$, which is not possible, since $d(b)=4$. This contradiction proves that if $u$ is adjacent to $a$, then it must be adjacent to both $b$ and $c$.
	
	Suppose, on the other hand, $u$ is adjacent to $b$ and hence to $c$. We prove that $u$ must be adjacent to $a$. For otherwise, since all geodesics in  $G$ are of length less than 3, we can find a $u-a$ geodesic of length 2, say $uxa$. Clearly $x\neq v, b, c, d$ and $x$ cannot be adjacent to $v, b, c, d$ and the situation is similar to the one in Figure~\ref{u must be adj to b}, with the roles of $x$ and $a$ interchanged  and hence will lead to a contradiction. Hence $u$ must be adjacent to $a$.
	
	Hence, in any case, $u$ must be adjacent to $a$, $b$ and $c$. Thus we obtain a 2-stress regular subgraph $K$ of $G$ which is  isomorphic to the graph given by the Figure~\ref{stress regular not regular}. Note that if $G$ contains any other vertex $x$, then since $G$ is connected, there must  exist a vertex $y$ which is adjacent to at least one vertex in $K$. But then, by Lemma~\ref{stress imposing vertex}, it follows that $y$ must be adjacent to each vertex in $K$ and hence $G$ will contain a vertex of degree 5, a contradiction. Hence we  have $V(G)=V(K)$. Clearly $K$ is the maximal subgraph with the vertex set $V(K)$. Hence it follows that $G=K$. 
	\vskip .3cm
	
	\noindent  \textbf{Case 2.} Suppose $d(v)=3$ for all $v\in V(G)$. Let $p$ be any vertex in $G$. Let the points adjacent to $p$ be $a$, $b$ and $c$. Since $\text{str}(p)=2$, we must have that exactly one pair of vertices among $a,b,c$ are adjacent to each other. Let $b$ and $c$ be adjacent to each other and $a$ be non-adjacent to both $b$ and $c$. Let $v$ and $w$ be the other two neighbors of $a$. Since $d(p)=3$, it follows that $p$ must be non-adjacent to both $v$ and $w$ and hence,  $v$ must be adjacent to $w$, since $\text{str}(a)=2$. The arrangement of vertices is as shown in the Figure~\ref{First situation during 2-stress-regular, 3 regular}. Note that all the vertices shown are distinct.
	
	\begin{figure}[h!]
		\unitlength 1mm 
		\linethickness{0.4pt}
		\ifx\plotpoint\undefined\newsavebox{\plotpoint}\fi 
		\begin{picture}(120,30)(-35,190)
		
		\put(14,194){\circle*{2}}
		\put(50,194){\circle*{2}}
		\put(14,214){\circle*{2}}
		\put(50,214){\circle*{2}}
		
		\put(14,194){\line(0,1){20}}
		\put(50,194){\line(0,1){20}}
		
		\put(24,204){\circle*{2}}
		\put(40,204){\circle*{2}}
		\put(24,204){\line(1,0){15}}
		
		\put(14,194){\line(1,1){10}}
		\put(50,194){\line(-1,1){10}}
		\put(14,214){\line(1,-1){10}}
		\put(50,214){\line(-1,-1){10}}
		
		\put(24,201){\makebox(0,0)[cc]{$p$}}
		\put(40,201){\makebox(0,0)[cc]{$a$}}
		\put(11,194){\makebox(0,0)[cc]{$c$}}
		\put(11,214){\makebox(0,0)[cc]{$b$}}
		\put(53,214){\makebox(0,0)[cc]{$v$}}
		\put(53,194){\makebox(0,0)[cc]{$w$}}
		
		\end{picture}
		\caption{A figure to illustrate a situation in Case 2}
		\label{First situation during 2-stress-regular, 3 regular}
	\end{figure}
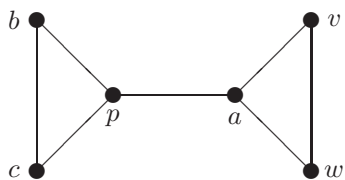
	
	Now we claim that if $b$ is adjacent to $v$, then $c$ must be adjacent to $w$. On the contrary, suppose $b$ is adjacent to $v$ and $c$ is not adjacent to $w$. Now since all geodesics in $G$ are of length less than 3, there must be a $c-w$ geodesic of length 2, say $cxw$. The present situation is shown in the Figure~\ref{Second situation during 2-stress-regular, 3 regular}.
	
	\begin{figure}[h!]
		\unitlength 1mm 
		\linethickness{0.4pt}
		\ifx\plotpoint\undefined\newsavebox{\plotpoint}\fi 
		\begin{picture}(120,30)(-35,190)
		
		\put(14,194){\circle*{2}}
		\put(50,194){\circle*{2}}
		\put(14,214){\circle*{2}}
		\put(50,214){\circle*{2}}
		
		\put(14,194){\line(0,1){20}}
		\put(14,194){\line(1,0){36}}
		\put(50,194){\line(0,1){20}}
		\put(14,214){\line(1,0){36}}
		
		\put(24,204){\circle*{2}}
		\put(40,204){\circle*{2}}
		\put(24,204){\line(1,0){15}}
		
		\put(14,194){\line(1,1){10}}
		\put(50,194){\line(-1,1){10}}
		\put(14,214){\line(1,-1){10}}
		\put(50,214){\line(-1,-1){10}}
		
		\put(24,201){\makebox(0,0)[cc]{$p$}}
		\put(40,201){\makebox(0,0)[cc]{$a$}}
		\put(11,194){\makebox(0,0)[cc]{$c$}}
		\put(11,214){\makebox(0,0)[cc]{$b$}}
		\put(53,214){\makebox(0,0)[cc]{$v$}}
		\put(53,194){\makebox(0,0)[cc]{$w$}}
		
		\put(32,194){\circle*{2}}
		\put(32,191){\makebox(0,0)[cc]{$x$}}
		
		\end{picture}
		\caption{A figure to illustrate a situation in Case 2}
		\label{Second situation during 2-stress-regular, 3 regular}
	\end{figure}
	
	Note that $x\neq p$ since $p$ is not adjacent to $w$ and $x \neq a$ since $c$ cannot be adjacent to $a$. Suppose $x=v$, then $c$ would be adjacent to $v$, but then $v$ will have degree 4, a contradiction. Hence $x\neq v$. Similarly $x \neq b$. Further, since $d(u)=3$ for all $u \in V(G)$, we have that $x$ cannot be adjacent to $p, a, b, v$. Let $y \neq c, w$ be the other neighbor of $x$. Since $x$ is not adjacent to $p, a, b, v$, we have $y \neq p, a, b, v$. Now since $\text{str}(x)=2$, by Lemma~\ref{stress imposing vertex} it follows that $y$ must be adjacent to either $c$ or to $w$, both of which are not possible, since $d(c)=d(w)=3$. Hence if $b$ is adjacent to $v$, then $c$ must be adjacent to $w$. Similarly, if $c$ is adjacent to $w$, then $b$ must be adjacent to $v$. Thus we get that $b$ is adjacent to $v$ if and only if $c$ is adjacent to $w$. Similarly $b$ is adjacent to $w$ if and only if $c$ is adjacent to $v$.
	
	Now we prove that  $b$ is adjacent to  either $v$ or to $w$. Suppose not. Then from the above part, it follows that $c$ cannot be adjacent to $v$ and $w$. Now  since, all geodesics in $G$ are of length less than 3, there must be a $b-v$ geodesic of length 2, say $bxv$. The present situation is as shown in the Figure~\ref{Third situation during 2-stress-regular, 3 regular}.
	
	\begin{figure}[h!]
		\unitlength 1mm 
		\linethickness{0.4pt}
		\ifx\plotpoint\undefined\newsavebox{\plotpoint}\fi 
		\begin{picture}(120,30)(-35,190)
		
		\put(14,194){\circle*{2}}
		\put(50,194){\circle*{2}}
		\put(14,214){\circle*{2}}
		\put(50,214){\circle*{2}}
		
		\put(14,194){\line(0,1){20}}
		\put(50,194){\line(0,1){20}}
		\put(14,214){\line(1,0){36}}
		
		\put(24,204){\circle*{2}}
		\put(40,204){\circle*{2}}
		\put(24,204){\line(1,0){15}}
		
		\put(14,194){\line(1,1){10}}
		\put(50,194){\line(-1,1){10}}
		\put(14,214){\line(1,-1){10}}
		\put(50,214){\line(-1,-1){10}}
		
		\put(24,201){\makebox(0,0)[cc]{$p$}}
		\put(40,201){\makebox(0,0)[cc]{$a$}}
		\put(11,194){\makebox(0,0)[cc]{$c$}}
		\put(11,214){\makebox(0,0)[cc]{$b$}}
		\put(53,214){\makebox(0,0)[cc]{$v$}}
		\put(53,194){\makebox(0,0)[cc]{$w$}}
		
		\put(32,214){\circle*{2}}
		\put(32,217){\makebox(0,0)[cc]{$x$}}
		
		\end{picture}
		\caption{A figure to illustrate a situation in Case 2}
		\label{Third situation during 2-stress-regular, 3 regular}
	\end{figure}

	As earlier we have $x \neq p$ and $x \neq a$. Further, since $c$ is not adjacent to $v$, we have $x \neq c$ and since $b$ is not adjacent to $w$, we have $x\neq w$.  Again, since $b$ and $w$ are not adjacent,  it follows that there must be a $b-w$ geodesic, say $P$ of length 2. Now since $d(b)=3$ and since $w$ is not adjacent to $c$ and $p$, it follows that $P=bxw$ so that $x$ and $w$ are adjacent.	But then, since $c$ is not adjacent to any of the neighbors of $w$, it follows that there is no path of length less than 3 between $c$ and $w$, which is a contradiction.
	
	Hence  $b$ is adjacent to either $v$ or $w$. But since $d(b)=3$, it cannot be adjacent to both. We may assume $b$ is adjacent to $v$. But then, from the earlier part, it follows that $c$ must be adjacent to $w$ and hence we obtain a 2-stress regular subgraph $H$ of $G$ which is  isomorphic to the graph given by the Figure~\ref{2 stress regular 3 regualr}. Since $d(u)=3$ for all $u\in V(G)$ and $G$ is connected, it follows that $V(G)=V(H)$. Clearly $H$ is the maximal subgraph with the vertex set $V(H)$. Hence we must have $G=H$.
	
\end{proof}

\section*{Acknowledgments}

The authors are thankful to Prof. S. Arumugam, Kalasalingam University, Tamilnadu, India, for his valuable suggestions during his visit to Mangalore University on 16 March 2019.

\end{document}